\documentclass[12pt]{amsart}

\usepackage{amsmath,amssymb,amsthm}
\usepackage[mathscr]{eucal}

\newtheorem{theorem}{Theorem}

\theoremstyle{plain}

\newtheorem{lemma}{Lemma}

\theoremstyle{definition}

\theoremstyle{definition}

\def\D{\mathbb{D}}%
\def\T{\mathbb{T}}%
\def\H{\mathcal{H}}%
\def\leq{\leqslant}%
\def\geq{\geqslant}%
\def\sk{\smallskip}
\def\1{\mathbf 1}
\DeclareMathOperator{\kerr}{Ker}
\DeclareMathOperator{\imm}{Ran}

\DeclareMathOperator{\spann}{span}

\begin{document}
\title{Schmidt subspaces of Hankel operators }

\author{Maria T. Nowak, Pawe\l \ Sobolewski, and Andrzej So\l tysiak }

\address{Maria T. Nowak
\newline Institute of Mathematics
\newline Maria Curie-Sk{\l}odowska University
\newline pl. M.
Curie-Sk{\l}odowskiej 1
\newline 20-031 Lublin, Poland}
\email{maria.nowak@mail.umcs.pl}

\address{Pawe{\l} Sobolewski
\newline Faculty of Electrical Engineering and Computer Science
\newline Lublin University of Technology
\newline 38A Nadbystrzycka St.
\newline 20-618 Lublin, Poland}
\email{p.sobolewski@pollub.pl}

\address{Andrzej So{\l}tysiak
\newline Faculty of Mathematics and Computer Science
\newline Adam Mickiewicz University \newline ul. Uniwersytetu Pozna\'nskiego 4
\newline 61-614 Pozna\'n, Poland}
\email{asoltys@amu.edu.pl}



\thispagestyle{empty}
\begin{abstract} We consider bounded Hankel operators $H_{\psi}$ acting on the Hardy space $H^2$ to $L^2\ominus H^2$ and obtain results on the Schmidt subspaces $E^+_s(H_\psi)$ of such operators defined as the kernels of $ H_{\psi}^{\ast}H_{\psi}-s^2I$ where $s>0$. These  spaces have been recently studied in \cite{GP} and \cite{GP1} in the context of anti-linear Hankel operators. We also discuss the  range of the Hankel operators with symbols being the complex conjugates of functions in the unit ball of $H^{\infty}$.
\end{abstract}
\maketitle

\section{Introduction}
Let $H^2$ denote  the classical Hardy space of the unit disk $\D$ and let $\T=\partial \D$. Under the standard identification of a function  $ f\in H^2$ with its  nontangential boundary value function  in $L^2( \T)=L^2 $, also denoted by $f$, we  consider $H^2$ as  a closed subspace of $L^2$. Let $H_0^2$ denote the subspace of $H^2$ consisting of functions vanishing at $0$,
and $P_+$ and $P_- $   the orthogonal projections of $L^2$ onto $H^2$ and $L^2\ominus H^2=\overline{H_0^2}$, respectively,  where $\overline{H_0^2}$ consists of complex conjugates of the functions in $H_0^2$.

For $\varphi\in L^2$, we define the Hankel operator $H_{\varphi}f: H^2\to \overline{H_0^2} $  on the dense subset of polynomials in $H^2$ by  the formula
 \begin{equation}\label{hankel}
H_{\varphi}f= (I-P_+)(\varphi f)=P_-(\varphi f).
\end{equation}
It is worth noting  here that   Hankel operators of this type  are not only  studied on  the Hardy space but also  on different spaces of holomorphic function, see, e.g. \cite{ER}, \cite{HV}, \cite{DZ}.
However, in this paper we deal with classical Hardy space $H^2$.

It is known (see \cite [ Theorem 1.3] {P}, \cite [pp. 196--197] {Z})  that $H_{\varphi}$ is bounded on $H^2$ if and only if there exists a function $\psi\in L^{\infty}$
such that  $H_{\psi}=H_{\varphi}$. This implies that bounded Hankel operators may be also represented by   Hankel operators with symbols   being conjugates  of analytic functions of bounded mean oscillation (BMOA).

Throughout the paper we will use a conjugation operator  $C\colon\,L^2(\T)\rightarrow L^2(\T)$  defined by
\[
Cf(z)=\bar{z}\overline{f(z)}, \quad  z\in \T.
\]
It is known that $C$ is an anti-linear surjection satisfying the following relations:
\[
C^2=I,\quad  CP_+=P_-C, \quad\text{and }\quad CH^2=\overline{H_0^2}= L^2\ominus H^2, \quad \text{see, e.g., \cite [p. 336 ] {N1}.}
\]

For a Hankel operator $H_{\psi}$  and $s>0$, the Schmidt space $E_s^+(H_{\psi})$ is defined as the kernel of  the operator $H_{\psi}^{\ast}H_{\psi} -s^2I$. A real number $s>0$ is called a singular value of $H_{\psi}$ if $s^2$ is an eigenvalue of $H_{\psi}^{\ast}H_{\psi}$.

 In the recent papers
\cite{GP} and \cite {GP1} the authors described the Schmidt spaces for anti-linear Hankel operators $\widetilde H_{u}$ with the symbols $u$ being  BMOA functions. The operator
$\widetilde H_{u}$ is defined on $H^2$ by
\[
\widetilde H_{u}f= P_+(u\overline {f}).
\]
The Schmidt spaces $E_{\widetilde H_u}(s)$ for  anti-linear Hankel operators $\widetilde H_{u}$ are defined analogously, that is,
\[
E_{\widetilde H_u}(s)=\kerr(\widetilde H _u^2-s^2I),\quad s>0.
\]
The  main result in the above mentioned papers states that every Schmidt space $E_{\widetilde H_u}(s)$ is of the form $pK_{\theta}$, where $\theta$ is an inner function, $K_{\theta}= H^2\ominus\theta H^2$ is the corresponding model space and $p\in H^2$ is an isometric multiplier on $K_{\theta}$.

We  claim  that for $u\in \text{ BMOA}$ the Schmidt spaces $E_{\widetilde H_u}(s)$ and  $E^+_s(H_{\overline{zu}})$  coincide. To see this
notice that if $C$ is the conjugation defined above, then
\[
C H_{\overline{uz}}(f)=CP_-(\overline{uz}f)=P_+C(\overline{uz}f)=P_+(u\bar{f})=\widetilde{H}_uf
\]
and
\[
H^\ast_{\overline{uz}}C(f)=
P_+(uz C(f))=P_+(u\bar{f})=\widetilde H_uf.
\]
Consequently,
\[
H^\ast_{\overline{uz}}H_{\overline{uz}}=H^\ast_{\overline{uz}}C\cdot C H_{\overline{uz}}=\widetilde{H}^2_u.
\]

For $\chi\in
L^{\infty}$ the Toeplitz operator $T_{\chi}$ on $H^2$ is given by
$T_{\chi}f=P_{+}(\chi f)$.
Let $S=T_z$ denote the unilateral shift on $H^2$. Then
 $S^*=T_{\bar z}$ is the so called   backward shift operator on $H^2$.
 A closed subspace $M$ of $H^2$ is said to be nearly
$S^*$-invariant if
\[
(f\in M, \ f(0)=0)\implies  S^*f\in M.
\]
Nearly $S^*$-invariant spaces are characterized by
 Hitt's Theorem (\cite{H}, \cite{S1}, \cite{S2}, \cite{FM}).

\sk

\textit{The nontrivial  closed subspace M of $H^2$ is
  nearly $S^*$-invariant  if and only if
there exists a function $f\in\ M$   of unit norm such that $f(0)>0$   and an $S^\ast$-invariant subspace $M'$  on which  $T_f$
acts isometrically and such that $M= T_fM'$.}

\sk

It is  known that the kernels of Toeplitz operators are  nearly
$S^*$-invariant but not every nearly   $S^*$-invariant space is the kernel of a Toeplitz operator (\cite{S2}, \cite{FM}).

In Section 3, we show that if $s=\|H_{\psi}\|$, then $E^+_s(H_{\psi})$ is actually the kernel of a Toeplitz operator. Consequently, Lemma 7.1 in \cite{GP} follows from the  known properties of kernels of Toeplitz operators (see e.g. \cite{S2}, \cite [Chapter 30]{FM}).
We also  prove, in our setting, that if   not all functions in the space $E^+_s(H_{\psi})$ vanish at 0, then this space is nearly $S^\ast$- invariant.  In Theorem 5 we  show that if all functions in  $E^+_s(H_{\psi})$ have zeros at $0$  of order $n$, but not of order $n+1$, then  $E^+_s(H_{\psi})=z^nN$, where $N$ is a nearly $S^\ast$- invariant subspace of $H^2$. The proof of Theorem 5  is different from the proof of the analogous result for anti-linear Hankel operators given in \cite[Subsection 2.5]{GP1}.

In the next  section we also discuss  the range  of  $H_{\bar b}$, where $b$ is a function from the unit ball of $H^{\infty}.$

\section
{Remarks on the range of some Hankel operators}
It is well known  that   a  Hankel operator  is a partial isometry if and only if it has the form $H_{\bar{\theta}}$ where $\theta$ is an inner function. The initial space of $H_{\bar{\theta}}$ is the model space $K_{\theta}= H^2\ominus \theta H^2$ and the final space of $H_{\bar{\theta}}$
is
\begin{equation}
\overline{H_0^2}\ominus\bar{\theta}\overline{H^2_0}
=\bar z\overline{K_{\theta}}.
\label{range}
\end{equation}
It is also  worth noting   that if the kernel of  $ H_{\psi}$ is trivial, then the range of $H_{\psi}$ cannot be closed in $\overline {H_0^2}$,  because  for any $ \psi\in L^{\infty}, $  $\|H_{\psi}z^n\|\rightarrow 0 $, as $n\rightarrow \infty$.
The Hankel operators with closed ranges are characterized by  Theorem 2.8 in \cite{P}. They are exactly Hankel operators with symbols $\psi=\bar{\theta}\varphi$, where
$\theta$ is an inner function,  $ \varphi\in H^{\infty}$,  and  $\theta$ and $\varphi$ satisfy the corona condition.

 Observe that (\ref{range}) means  that $CH_{\bar\theta}$ maps $H^2$ onto $K_{\theta}$. In Theorem 1  we show that for $b$ in the unit ball of $H^{\infty}$  the range of  $C H_{\bar b}$ is contained in  the de Branges-Rovnyak space $\H(b)$, and when  $b$ is an extreme point of the unit ball of $H^{\infty}$, $\imm H_{b}$ is  dense in  $\H(b)$.

Recall that for $b$ in the unit ball of $H^{\infty}$, the de Branges-Rovnyak space $\H(b)$ is the image of $H^2$ under the operator $(I-T_bT_{\bar b})^{1/2}$ with the  range norm $\|\cdot\|_b$
defined by
\[
\|(I-T_bT_{\bar b})^{1/2}f\|_b=\|f\|_2\qquad \text{for} \ \ f\in H^2\ominus \kerr(I-T_bT_{\bar b})^{1/2}.
\]
Analogously, the space $\H(\bar b)$ is defined as the image of $H^2$  under  $(I-T_{\bar b}T_b)^{1/2}$ with the corresponding range norm $\|\cdot\|_{\bar b}$. Proofs of the stated below properties of these spaces can be found in \cite {S} and \cite {FM}.

It is known that $\H(b)$ is a Hilbert space with
reproducing kernel
\[
k_w^b(z)=\frac{1-\overline{b(w)}b(z)}{1-\bar{w}z}\quad(z,w\in\D).
\]
If $b$ is inner, then $\H(b)=K_{b}$. Moreover, $\H(\bar b)=\{0\}$ if and only if $b$ is an inner function.
We also mention that $\H(b)$ is a dense subset of $H^2$ in $\| \cdot\|_2$  norm if and only if $b$ is not inner and $H^2= \H(b)$ if and only if $\|b\|_{\infty}<1$.

Many  properties  of  $\H(b)$ depend  on whether $b$ is or is not extreme point of the closed unit ball of $H^{\infty}$.  It is known that $b$ is an extreme  point of the unit ball of $H^{\infty}$  if and only if
\[
\int_{\T}\log(1-|b(e^{it})|^2)\,dt=-\infty.
\]

We will simply say that $b$ is  an extreme or nonextreme function. Clearly, every inner function is extreme.

\begin{theorem}
If  $b$ is a function in the closed unit ball of $H^\infty$, then $CH_{\bar b}$ maps $H^2$ into  $\H(b)$.
Moreover, if $b$ is extreme, then  $CH_{\bar b}(H^2)$ is dense  in $\H(b)$.
\end{theorem}

\begin{proof}
To prove the first statement we use the following characterization of  the space $\H(b)$ (see \cite[I-8]{S}, \cite[Thm. 17.\,8]{FM}): if   $g\in H^2$, then
\[
g\in\H(b) \iff T_{\bar{b}}g\in \H(\bar{b})
\]
and for $ g\in\H(b)$,
\begin{equation}
\|g\|^2_b=\|g\|^2_2+\|T_{\bar{b}}g\|^2_{\bar{b}}.\label{norm}
\end{equation}
For  $f\in H^2$ set
\[
g=CH_{\bar b}f=CP_-(\bar{b}f)=P_+C(\bar{b}f)=P_+(b\bar{z}\bar{f}).
\]

Then   for any $h\in H^2$,
\begin{align*}
\langle T_{\bar{b}}g,h\rangle&=\langle T_{\bar{b}}P_+(b\bar{z}\bar{f}),h\rangle=\left\langle P_+\left(\bar{b}P_+(b\bar{z}\bar{f})\right),h\right\rangle=
\langle \bar{b}P_+(b\bar{z}\bar{f}),h\rangle\\
\noalign{\sk}
&=\langle P_+(b\bar{z}\bar{f}),bh\rangle
=\langle b\bar{z}\bar{f},bh\rangle=\langle |b|^2\bar{z}\bar{f},h\rangle=\left\langle P_+\left(|b|^2\bar{z}\bar{f}\right),h\right\rangle
\\
\noalign{\sk}
&=-\left\langle P_+\left((1-|b|^2)\bar{z}\bar{f}\right),h\right\rangle+\langle P_+(\bar{z}\bar{f}),h\rangle=
-\left\langle P_+\left((1-|b|^2)\bar{z}\bar{f}\right),h\right\rangle.
\end{align*}
Thus
\[
T_{\bar{b}}g=-P_+\left((1-|b|^2)\bar{z}\bar{f}\right).
\]

Since  the space $\H(\bar b)$  consists of  the Riesz projections of functions from the weighted space $L^2(1-|b|^2)$  of measurable functions $h$ on  $ \T$ such that $\int_{\T}(1-|b(e^{i\theta})|^2)|h(e^{i\theta})|^2d\theta<\infty$ (\cite[Thm. 25.1, Cor. 25.2]{FM},  \cite[III-2 ]{S}), we see that $T_{\bar g}f\in\H(\bar{b})$ and  our claim is proved.

Now assume that $b$ is extreme. If $b$ is an inner function, then by (\ref {range}),  $CH_{\bar{b}}(H^2)=K_b=\H(b)$.

Now we want to identify the images of  $k_\lambda(z)=(1-\bar{\lambda}z)^{-1}$, the reproducing kernels for $H^2$, under $CH_{\bar b}$.  We have
\begin{align*}
\left(CH_{\bar{b}}k_\lambda\right)(z)&=C\left(\overline{b(z)}k_\lambda(z)-P_+(\bar{b}k_\lambda)(z)\right)=C\left(\overline{b(z)}k_\lambda(z)-\overline{b(\lambda)}k_\lambda(z)\right)\\
\noalign{\sk}
&=C\left(\frac{\overline{b(z)}-\overline{b(\lambda)}}{1-\bar{\lambda}z}\right).
\end{align*}
Thus for $|z|=1$,
\begin{equation}\label{four}
\left(CH_{\bar{b}}k_\lambda\right)(z)=C\left(\frac{\bar{z}\left(\overline{b(z)}-\overline{b(\lambda)}\right)}{\bar{z}-\bar{\lambda}}\right)=\dfrac {b(z)-b(\lambda)}{z-\lambda}=Q^b_\lambda(z).
\end{equation}
It is known that in the case $b$ is extreme, the family of the difference quotients $\{Q^b_\lambda: \lambda \in\D\}$  is complete in  $ \H ( b)$ (\cite[Cor. 26.\,18]{FM}), that is
\[\H(b)=\spann\{Q^b_\lambda \colon\,\lambda\in\D\},
\]
where the closure is taken in $\H(b).$

\sk

Now assume that $h\in \H(b)$. Then
$h=\lim\limits_{n\to\infty}h_n$, where $h_n=\alpha_1Q^b_{\lambda_1}+\ldots+\alpha_nQ^b_{\lambda_n}$,    $\alpha_k\in \mathbb C$, $k=1,\dots, n$.
Since  in view of (\ref{four}), $h_n=CH_{\bar{b}}(f_n)$, where $f_n=\alpha_1k_{\lambda_1}+\ldots+\alpha_nk_{\lambda_n}$ we have
$h=\lim\limits_{n\to\infty}CH_{\bar{b}}(f_n)$. This proves the density  of $\imm CH_{\bar b}$ in $\H(b)$.
\end{proof}

\section{Schmidt spaces of Hankel operators}

 For a singular value $s$ of $H_{\psi}$,  a pair of functions $ (f,g)$, $f\in H^2$, $g\in \overline{H^2_0}$, is called an \emph{$s$-Schmidt pair} or  (\emph{$s$-pair}) of $H_{\psi}$ if  $H_{\psi}f=sg$ and  $H_{\psi}^{\ast}g=sf$.
We note that $f\in E_s^+(H_{\psi})$ if and only if $\left(f,\frac{1}{s}H_{\psi}f\right)$ is an $s$-Schmidt pair.

It is known that if $(f,g)$ is an $s$-pair, then $\psi_s=\frac{g}{f}$ is a unimodular function that does not depend on  the choice of an $s$-pair $(f,g)$ (\cite{P}, \cite{N1}).

It is also defined $E_s^-(H_{\psi})$ as the kernel  of $H_{\psi}H_{\psi}^{\ast}- s^2I$. Clearly $(f,g)$ is an $s$-pair of $H_{\psi}$ if and only if $(g,f)$ is an $s$-pair of $H_{\psi}^{\ast}$.
We also  note that the conjugation $C$ maps $E_s^+(H_{\psi})$ onto $E_s^-(H_{\psi})$. Indeed, since $CH_{\psi}= H_{\psi}^{\ast}C$ ,
\[
H_{\psi}^{\ast}H_{\psi}f= H_{\psi}^{\ast}C\cdot CH_{\psi}f= CH_{\psi}CH_{\psi}f=CH_{\psi}H_{\psi}^{\ast}Cf,\quad  f\in H^2,
\]
we see that
\[
H_{\psi}^{\ast}H_{\psi}f=s^2 f\iff CH_{\psi}H_{\psi}^{\ast}Cf=s^2f \iff H_{\psi}H_{\psi}^{\ast}Cf=s^2Cf.
\]

The following theorem characterizes the Schmidt space in the case when $\|H_{\psi}\|=s$.

\begin{theorem}
Assume that $s=\|H_\psi\|>0$ is a singular value of $H_{\psi}$, $\psi \in L^{\infty}$. Then
\[
\kerr(H_\psi^\ast H_\psi-s^2I)=\kerr T_\varphi,
\]
where $\varphi=\frac{H_\psi f}{f}$ and  $f$ is any nonzero function in $\kerr(H_\psi^\ast H_\psi-s^2I)$.
\end{theorem}

\begin{proof}
We start the proof with the observation that  under  the assumption $\|H_\psi\|=s>0$ the following conditions are equivalent for $f\in H^2$:
\begin{enumerate}
\item[{\rm(i)}]
$f\in\kerr(H_{\psi}^\ast H_{\psi}-s^2I)$;
\item[{\rm(ii)}]
$\|H_\psi f\|=s\|f\|$.
\end{enumerate}
Indeed, if $f\in\kerr(H_{\psi}^\ast H_{\psi}-s^2I)$, then
\[
\|H^\ast_\psi H_\psi f\|=s^2 \|f\|,
\]
which implies
\[
 s^2 \|f\|\leq \|H_{\psi}^\ast\|\|H_{\psi}f\|=s\|H_{\psi}f\|.
\]
Since
\[
\|H_{\psi}f\|\leq s\|f\|,
\]  (ii) holds true.

To prove the other implication it enough to notice that the operator $s^2I- H_{\psi}^\ast H_{\psi}$ is positive, or in other words
  for any $g\in H^2$,
\[
\langle (s^2I- H_{\psi}^\ast H_\psi )g,g\rangle\geq 0
\]
 and (ii) implies that
\begin{equation}
\langle (s^2I-H_\psi^{\ast}H_\psi) f,f\rangle=0.\label{iii}
\end{equation}
Now, by  the Cauchy-Schwarz inequality,
\[
|\langle (s^2I-H_\psi^{\ast}H_\psi) f,g\rangle|^2\leq \langle (s^2I-H_\psi^{\ast}H_\psi) f,f\rangle \langle(s^2I-H_\psi^{\ast}H_\psi)g ,g\rangle,\quad g\in H^2,
\]
which together with (\ref{iii}) gives
\[
(s^2I-H_\psi^{\ast}H_\psi) f=0.
\]

Thus,  $f\in\kerr(H_\psi^\ast H_\psi-s^2I)$ with  $s=\|H_{\psi}\|$, if and only if the Hankel operator $H_{\psi}$ attains its norm on the unit ball of $H^2$ at $f/\|f\|$ .
Moreover,  we know that $\left(f, \frac{1}{s}H_{\psi} f\right)$ is an $s$-pair for $H_{\psi}$ and, consequently, $|\varphi|=s$ a.e. on $\T$.
Furthermore by \cite [Theorem 11.5]{FM} (or  by \cite [Theorem 1.4] {P}),  $ H_\psi =H_\varphi $.
Thus
\[
H_\psi^\ast H_\psi-s^2I=H_{\varphi}^\ast H_{\varphi}-s^2I=T_{\bar{\varphi}}H_\varphi-s^2I=T_{\bar{\varphi}}T_\varphi.
\]
Indeed, for  $f\in H^2$, we get
\[
T_{\bar{\varphi}}(\varphi f-P_+(\varphi f))-s^2f=P_+(s^2f)+T_{\bar{\varphi}}T_\varphi f-s^2f=T_{\bar{\varphi}}T_\varphi f.
\]
Since $\kerr (T_{\bar{\varphi}}T_\varphi)=\kerr T_\varphi$, the claim is proved.
\end{proof}

For  a general case we obtain the following:
 \begin{theorem}
For  $\psi\in L^{\infty} $ let  $H_{\psi}$ be a Hankel operator defined by (\ref{hankel}). Let $s>0$ be a singular value for $H_{\psi}$.
 Then there exist an inner function $\theta$ and $\varphi\in L^{\infty}$ such that
 \[
E^+_{s}(H_{\psi})\subset \theta\kerr T_{\varphi}\subset E^+_{s}(H_{s\psi_s}),
\] where $\psi_s$ is defined at the beginning of this section.
\end{theorem}

In the proof of this theorem we apply some ideas coming from the proof of the fundamental Adamyan-Arov-Krein Theorem \cite {AAK} that states that the best  rank $n$ approximation of the Hankel operator $H_{\psi}$ is the same as its best approximation by rank $n$ Hankel operators (see also \cite{P}, \cite{N1}).

The following lemma is actually contained in Lemma 7.2.6 in \cite{N1} but we include its short proof for the reader's convenience.

\begin{lemma}
Assume that $\theta$ is an inner function and $f\in H^2$ are such that $f\theta\in E^+_{s}(H_{\psi})$.
Then $f\in E^+_{s}(H_{\theta\psi})$.
\end{lemma}
\begin{proof} If
\[ H_\psi^\ast H_\psi(\theta f)=s^2\theta f,
\]
then for any $g\in H^2$,
\begin{align*}  \langle H_{\theta \psi}^\ast H_{\theta \psi} f,g\rangle &= \langle\bar \theta \bar \psi P_{-}(\theta \psi f),g\rangle= \langle \bar\psi P_{-}(\theta \psi f),\theta g\rangle\\&=
\langle P_{+} (\bar \psi P_{-}(\psi\theta f)),\theta g\rangle= \langle H_{ \psi}^\ast H_{ \psi}(\theta f),\theta g\rangle=   \langle s^2\theta f,\theta g \rangle = \langle s^2 f,g \rangle.\end{align*}
\end{proof}

\begin{proof}[Proof of Theorem 3] We first show that
\[
E^+_{s}(H_{\psi})\subset \theta\kerr T_{\varphi}.
\]
First note that   if   $\psi_s$ is as specified above, then the Schmidt space  $E^+_s(H_{\psi})$  is a subspace of the Schmidt space  $E^+_s(H_{s\psi_s})$.
We observe that if $(f,g)$ is an $s$-pair for $ H_{\psi}$, then
\[
H_{s\psi_s}f = P_{-}\left(s\frac  g f f\right)= sP_{-}(g)=sg
\]
and since $|\psi_s| =1$ a.e. on $\T$,
\[
H_{s\psi_s}^{\ast}g= P_{+}\left(s\bar{\psi_s}g\right)= sP_{+}\left(\frac fg g\right)= sf.
\]
This shows that any   $s$-pair of $H_{\psi}$ is an $s$-pair of $H_{s\psi_s}$, which proves  the inclusion of the corresponding Schmidt spaces.
Hence
\[
H_{\psi}f=H_{s\psi_s}f\qquad \text{for} \ \ f\in  E^+_s(H_{\psi})
\]
or equivalently,
\[
(H_{\psi}-H_{s\psi_s})f= H_{\psi- s\psi_s}f=0.
\]
Since the kernel of a Hankel operator is the $S$-invariant subspace of $H^2$, $\kerr  H_{\psi- s\psi_s}= \theta H^2$, for some inner function $\theta$. This also means that $E^+_s(H_{\psi})\subset \theta H^2$.
Moreover, for   $ f\in  H^2$    the following  equalities hold
\begin{equation}
H_{\theta \psi}f=H_{\psi}\theta f=H_{s\psi_s}\theta f= H_{s\theta\psi_s}f.\label{theta}
\end{equation}
These equalities and Lemma 1 yield
\[
\theta f\in E^+_s(H_{\psi})\implies f\ \in E^+_s(H_{\theta \psi})= E^+_s(H_{\theta s\psi_s}).
\]
Now,  the reasoning analogous to that used in the proof of Theorem 2 gives
\[
\kerr (H_{s\theta \psi_s}^{\ast}H_{s\theta\psi_s}-s^2I)= \kerr T_{\theta \psi_s}.
\]

To prove the other inclusion assume that   $h\in  \kerr T_{\theta \psi_s}$.
Then $h\theta\psi_s\in \overline {H^2_0}$, and consequently,

\[
H_{ s\psi_s}(h\theta)= P_{-}(s\psi_sh\theta)=sh\theta\psi_s.
\]
Since $\psi_s$ is a unimodular function,  we get
\[
H_{s\psi_s}^{\ast}H_{s\psi_s}(\theta h)= s^2P_{+}(\bar{\psi_s}h\theta\psi_s)= s^2\theta h,
\]
which proves that  $\theta \kerr T_{\varphi}\subset E^+_{s}(H_{s\psi_s})$, where $\varphi=\theta\psi_s$.
\end{proof}

In our next theorem we  prove that if the space $ E^+_s(H_{\psi})$ is not orthogonal to the one-dimensional subspace of constant functions ($ E^+_s(H_{\psi})\not\perp 1$), then it  is nearly $S^\ast$-invariant.
In view of our remark in the Introduction this theorem is actually a restatement of the result obtained
in \cite{GP} and \cite {GP1} for anti-linear Hankel operators.
We present its proof in our setting in connection with Theorem 5 which describes a general  case of Schmidt subspaces.

\begin{theorem}  For  $\psi\in L^{\infty} $ let  $H_{\psi}$ be a Hankel operator defined by (\ref{hankel}). If  $s>0$ is  a singular value for $H_{\psi}$ and $E^+_s(H_{\psi})\not\perp 1$, then $E^+_s(H_{\psi})$ is a nearly $S^\ast$-invariant subspace of $H^2$.
\end{theorem}

\begin{proof}
We first show that for $f\in H^2$,
\begin{equation}\tag{$\ast$}
S^\ast H_{\psi}^{\ast}H_{\psi}Sf=H_{\psi}^{\ast} H_{\psi}f-\langle  z\psi f, 1\rangle P_+(\bar{\psi}\bar{z}).\label{i}
\end{equation}
To this end observe that for a $g\in H^2$,
\begin{align*}
\langle S^\ast H_{\psi}^{\ast} H_{\psi}Sf,g\rangle&=\langle H_{\psi}^{\ast}H_{\psi}(zf),zg\rangle=\langle P_+(\bar \psi P_{-}(\psi zf)),zg\rangle=\langle \bar \psi P_-(\psi zf),zg\rangle\\[4pt]
&=\langle \psi zf,P_-(\psi zg)\rangle
=\langle \psi f,\bar{z}P_-(\psi zg)\rangle=\langle P_-(zP_-(\psi f)),\psi zg\rangle.
\end{align*}
Since
\begin{align*}
P_-(zP_-(\psi f))&=P_-\biggl(z\sum\limits_{n=1}^\infty\langle\psi f,\bar{z}^n\rangle \bar{z}^n\biggr)=P_-\biggl(\langle\psi f,\bar{z}\rangle 1+\sum\limits_{n=2}^\infty\langle\psi f,\bar{z}^n\rangle \bar{z}^{n-1}\biggr)\\[4pt]
&=zP_-(\psi f)-\langle z\psi f, 1\rangle,\end{align*}
we get
\begin{align*}
\langle S^\ast H_{\psi}^{\ast} H_{\psi}Sf,g\rangle&=\langle  zP_-(\psi f)-\langle z\psi f, 1\rangle, \psi zg\rangle= \langle \bar \psi P_-(\psi f) - \bar{z}\bar{\psi}\langle  z\psi  f,1\rangle, g\rangle\\[4pt]&=
 \langle  H_{\psi}^{\ast} H_{\psi} f-\langle  z\psi  f,1\rangle P_+(\bar{z}\bar{\psi}), g\rangle
 \end{align*}
which proves (\ref{i}).
 Replacing  $f$ by $S^{\ast}f$ in (\ref{i}) we get
\begin{equation}
S^\ast H_{\psi}^{\ast}H_{\psi}S S^{\ast}f=H_{\psi}^{\ast} H_{\psi}S^\ast f- \langle  z\psi S^\ast f, 1\rangle P_+(\bar{\psi}\bar{z}).\label{star}
\end{equation}
Furthermore,  the identity
\[
SS^\ast=I-\langle\cdot, 1\rangle 1,
\]
implies
\begin{equation}
S^\ast H_{\psi}^{\ast}H_{\psi}S S^{\ast}f= S^\ast H_{\psi}^{\ast}H_{\psi}f -\langle f,1\rangle S^{\ast}H_{\psi}^{\ast}H_{\psi}1.\label{star1}
\end{equation}
It then follows from (\ref{star}) and (\ref{star1}) that
\begin{equation}
S^\ast H_{\psi}^{\ast}H_{\psi}f- H_{\psi}^{\ast} H_{\psi}S^\ast f= - \langle z\psi S^\ast f,1\rangle P_+(\bar{\psi}\bar{z}) +\langle f,1 \rangle  S^\ast H_{\psi}^{\ast} H_{\psi} 1.   \label{ii}
\end{equation}
Now notice that for $f,h\in  E^+_s(H_{\psi})$  we have
\begin{align*}
\langle S^\ast H_{\psi}^{\ast}H_{\psi}f- H_{\psi}^{\ast} H_{\psi}S^\ast f, h\rangle &=  \langle S^\ast H_{\psi}^{\ast}H_{\psi}f,h\rangle - \langle H_{\psi}^{\ast} H_{\psi}S^\ast f, h\rangle\\ &= \langle  s^2 S^\ast f , h\rangle - \langle  s^2 S^\ast f, h\rangle= 0.
\end{align*}
So, if we assume that $f\in E^+_s(H_{\psi}) $ is such that $f(0)=0$, then the last equality and   (\ref{ii}) imply  that for any $h\in  E^+_s(H_{\psi})$,
\[
\langle  z\psi S^\ast f,1\rangle \langle  P_+(\bar{\psi}\bar{z}), h\rangle =\langle  \psi  f,1\rangle \langle  P_+(\bar{\psi}\bar{z}), h\rangle=0.
\]
Now we are going to prove that if $f$ is as above, then
\[
\langle  \psi f,1\rangle =0.
\]
Let  $h\in E^+_s(H_{\psi})$ be such that  $h(0)\ne0$ and put   $g=H_{\psi}h$. Then $g\in E^-_s(H_{\psi})$,
  $h_1=Cg\in E^+_s(H_{\psi})$  and we get
\begin{align*}
 \langle  P_+(\bar{\psi}\bar{z}), h_1\rangle &=\langle CP_-(\psi),h_1\rangle=\langle CP_-(\psi),Cg\rangle\\[4pt]
&=\langle g,P_-(\psi)\rangle
=\langle H_{\psi}h, H_{\psi} 1\rangle=\langle H_{\psi}^\ast H_{\psi}h, 1\rangle=
s^2\langle h,1\rangle\ne 0.
\end{align*}

$S^\ast$-invariance of $E^+_s(H_\psi)$ follows from (\ref{star}).
\end{proof}

If all functions in a nontrivial Schmidt space  $E^+_s(H_\psi)$ vanish at zero, then   there is  a positive integer $n$ such that $E^+_s(H_\psi)= z^n N$, where $N$ is a subspace of $H^2$ and $N\not\perp 1$.

In this case we get the following theorem.
\begin{theorem}
If  $s>0$ is  a singular value for $H_{\psi}$ and $E^+_s(H_{\psi})=z^n N $ where $N\not\perp 1$, then $N$  is a nearly $S^\ast$-invariant subspace of $H^2$.
\end{theorem}

\begin{proof}
We  first observe that the  nearly $S^\ast$-invariance of $N$ is equivalent to the implication: if  $f\in E^+_s(H_\psi)$ is such that $f^{(k)}(0)=0$ for $k=0,1,\dots, n$,  then  $S^{\ast}f\in E^+_s(H_\psi)$.
Indeed,  such an $f$ can be written as $f= z^nf_1=S^nf_1$, where $f_1\in N$ and $f_1(0)=0$ and nearly $S^\ast$-invariance of $N$ means that $S^{\ast}f_1\in N$, or in other words,
\[
z^nS^\ast f_1 = S^nS^{\ast}f_1= S^{n-1}f_1=S^{\ast}f\in E^+_s(H_{\psi}).
\]

Next notice  that it follows from Lemma 1 that $N\subset  E^+_s(H_{z^n\psi})$. Moreover, the assumptions imply that $ E^+_s(H_{z^n\psi}) \not\perp 1$ and, by the previous theorem, $ E^+_s(H_{z^n\psi})$ is nearly $S^\ast$- invariant. Furthermore it follows
 from the proof of this  theorem that
\begin{equation}
 0= \langle \psi z^n f_1, 1\rangle=\langle \psi  f, 1\rangle\label{zero}
\end{equation}
for any $f_1\in N $  such that $f_1(0)=0$.

\sk

 Now our aim is to show that if $f\in E^+_s(H_\psi)$ is such that $f^{(k)}(0)=0$ for $k=1,\dots, n$, then the  function $h(z)=H_{\psi}^\ast H_{\psi}S^{\ast}f(z)$ satisfies the condition $h^{(k)}(0)=0$ for $k=0,1,\dots, n-1$. To this end note that if $\langle f,z^k\rangle=0$, then
\begin{align*}
0&=  s^2\langle f,z^k\rangle = \langle H_{\psi}^\ast H_{\psi}f ,z^k\rangle=  \langle P_-(\psi f), \psi z^k\rangle= \langle \bar z\psi  f,  \bar zP_-(\psi z^k)\rangle\\
&=\langle  P_-(\bar z\psi  f),
\bar z(\psi z^k-P_+(\psi z^k)\rangle=\langle\bar \psi  P_-(\bar z\psi  f), z^{k-1}\rangle -\langle z P_-(\bar z\psi  f),P_+(\psi z^k) \rangle\\
&=\langle H_{\psi}^\ast H_{\psi} S^\ast f, z^{k-1}\rangle -\langle \langle\bar z\psi f, \bar z  \rangle 1 ,P_+(\psi z^k) \rangle=\langle H_{\psi}^\ast H_{\psi} S^\ast f, z^{k-1}\rangle - \langle \psi f,1\rangle\langle 1, P_+(\psi z^k) \rangle.
\end{align*}
Since by (\ref{zero})  the last term is zero,  our claim is proved.

\sk

Now observe that if $f\in  E^+_s(H_{\psi})$ and $f_1$ are as above,  then the nearly $S^\ast$-invariance of $E^+_s(H_{z^n\psi})$ yields
\begin{align*} s^2S^\ast f_1&= H_{z^n\psi}^\ast H_{z^n\psi}S^\ast f_1= P_+(\bar{z}^n\bar{\psi}P_{-}(
{z^n\psi} S^\ast f_1))\\
&= S^{\ast n}P_+(\bar\psi P_{-}(\psi S^\ast f))
=S^{\ast n}H_{\bar\psi} ^\ast H_{\psi}S^{\ast}f.
\end{align*}

Since the function $h=H_{\bar\psi} ^\ast H_{\psi}S^{\ast}f $  has zero at least  of order $n$ at 0, multiplying the equality
\[
s^2S^*f_1= S^{\ast n}H_{\bar\psi} ^\ast H_{\psi}S^{\ast}f
\]
by $S^n$ from the left we arrive at
\[ s^2 S^\ast f=
H_{\psi}^\ast H_{\psi}S^\ast f,
\]
which ends the proof.
\end{proof}

\end{document}